\documentclass{article}

\usepackage{arxiv}

\usepackage[utf8]{inputenc} 
\usepackage[T1]{fontenc}    
\usepackage{url}            
\usepackage{booktabs}       
\usepackage{amsfonts}       
\usepackage{nicefrac}       
\usepackage{microtype}      
\usepackage{amsfonts}
\usepackage{amsthm}
\usepackage{amsmath}
\usepackage{amscd}
\usepackage{graphicx}
\usepackage{amssymb}
\usepackage{color}
\usepackage{tikz}
\usepackage[numbers]{natbib}
\usepackage[ruled,vlined]{algorithm2e}
\usetikzlibrary{positioning}
\usetikzlibrary{arrows}
\newtheorem{remark}{Remark}

\newtheorem{theorem}{Theorem}[section]

\usepackage{multirow}
\usepackage{tikz,bm}
\usetikzlibrary{arrows,shapes,chains}

\title{A stochastic homotopy tracking algorithm for parametric systems of nonlinear equations}

\author{
  Wenrui Hao\\
 Department of Mathematics\\
Pennsylvania State University\\
University Park, PA 16802, USA \\
  \texttt{wxh64@psu.edu} \\
   \And
Chunyue Zheng \\
 Department of Mathematics\\
Pennsylvania State University\\
University Park, PA 16802, USA \\
\texttt{cmz5199@psu.edu} \\
}

\begin{document}
\maketitle

\begin{abstract}
The homotopy continuation method has been widely used in solving parametric systems of nonlinear equations. But it can be very expensive and inefficient due to singularities during the tracking even though both start and end points are non-singular. The current tracking algorithms focus on the adaptivity of the stepsize by estimating the distance to the singularities but cannot avoid these singularities during the tracking. We present a stochastic homotopy tracking algorithm that perturbs the original parametric system randomly each step to avoid the singularities. We then prove that the stochastic solution path introduced by this new method is still closed to the original solution path theoretically. Moreover, several homotopy examples have been tested to show the efficiency of the stochastic homotopy tracking method.
\end{abstract}


\section{Introduction}
The homotopy continuation method is the main tool to solve systems of polynomial equations in numerical algebraic geometry (NAG) \cite{bates2013numerically,leykin2011numerical,wampler2005numerical}.
The basic idea is to trace out a one-real-dimensional solution curve described implicitly by a system of equations:
given a nonlinear system ${\bm F} ({\bm u})$ to solve, one first forms a nonlinear system ${\bm G} ({\bm u})$ that is related to ${\bm F} ({\bm u})$ in a prescribed way but has known, or
easily computable  solutions. The systems ${\bm G} ({\bm u})$ and ${\bm F} ({\bm u})$ are combined to form
a homotopy, such as the linear homotopy
\begin{equation} {\bm H} ({\bm u},t) = {\bm F} ({\bm u}) (1-t) + t {\bm G} ({\bm u})=0, \label{homotopy1}
\end{equation}
where  ${\bm G} ({\bm u})$ is a start system with known solutions and ${\bm F} ({\bm u})$ is the target system we want to solve. Then solutions of ${\bm F} ({\bm u})=0$ can be solved by tracking $t$ from $1$ to $0$ via this linear homotopy.  In NAG, there is a well-developed theory on how to choose the start system ${\bm G} ({\bm u})$ to guarantee all the solutions of ${\bm F} ({\bm u})$ via this homotopy.
Furthermore, by constructing different start systems based on other theories, the homotopy continuation method has also successfully applied to compute solutions of nonlinear systems such as nonlinear PDEs \cite{HAOWENO,wang2018two,yang2018convergence}, machine learning \cite{CHao,hao2018equation}, and nonlinear systems in biology and physics \cite{hao2020spatial}. Moreover, the homotopy continuation method has been also used to explore the general parameter space, so-called
paramotopy,  as a quite powerful tool for many classes of problems that arise in practice \cite{bates2018paramotopy}.

In the linear homotopy setup, each solution path can be tracked via the prediction/correction algorithm \cite{bates2013numerically,leykin2011numerical,wampler2005numerical} which is referred as the homotopy tracking algorithm. This algorithm could become very inefficient if the parametric system is singular or near singular. To avoid the singular system,
in NAG \cite{wampler2005numerical},  the gamma trick is proposed to construct a random homotopy setup in (\ref{homotopy1})  by multiplying a random complex number. Then the probability of hitting a singularity during the tracking is zero. Nevertheless, the system could be still near singular so that the homotopy tracking is still time-consuming \cite{bates2013numerically,leykin2011numerical,wampler2005numerical}. In order to address this numerical challenge, an adaptive multi-precision path tracking algorithm \cite{bates2008adaptive} has been developed by adjusting precision in response to step failure according to the error estimates. An adaptive step-size homotopy tracking method \cite{HZ} has also been developed to control the tracking stepsize each time to compute the bifurcation point. An endgame algorithm \cite{bates2011parallel} has also been widely used to deal with the singularities at $t=0$. However, all these algorithms could be very slow and inefficient when the size of nonlinear systems becomes large \cite{bates2006bertini}.

Stochastic algorithms have been widely used in scientific computing \cite{cauwenberghs1993fast,nguyen2015iq}, e.g., the coordinate gradient descent has been developed for solving large-scale optimization problems \cite{nesterov2012efficiency} and has also been revised for solving the leading eigenvalue problem \cite{li2019coordinatewise}. Motivated by these stochastic algorithms, in this paper, we present an efficient stochastic homotopy tracking method that gives the original system a random perturbation each step so that it can avoid singularities and improve the efficiency during the tracking.  The paper is organized as follows: In section 2, we present a novel stochastic  homotopy tracking algorithm; In section 3, we analyze the stochastic homotopy tracking algorithm and show the solution path is close to the original solution path under certain conditions;
several numerical examples are presented in section 4 to illustrate the efficiency of the stochastic homotopy tracking method.

\section{Stochastic homotopy continuation method}
\label{sec:main}
Generally speaking, a nonlinear parametric system is written as $\mathbf{F}:
\mathbb{R}^n\times\mathbb{R}\rightarrow\mathbb{R}^n,$
\begin{equation}\label{Sys}
	\mathbf{F}(\mathbf{u},p)=\mathbf{0},
\end{equation}
where $p\in[a,b]$ is a parameter and $\mathbf{u}$ is the variable vector that depends on the parameter $p$, i.e., $\mathbf{u}=\mathbf{u}(p)$.
Suppose we have a solution at the starting point, namely $\mathbf{u}(a)=\mathbf{u}_0$, the homotopy tracking along the solution path, $\mathbf{u}(p)$, reduces down to solving the Davidenko differential equation \cite{bates2013numerically,wampler2005numerical},
\begin{equation}\label{Sys:trad}
	\left\{
	\begin{aligned}
		&	\mathbf{F}_\mathbf{u}(\mathbf{u},p)\frac{d \mathbf{u}}{dp}+\mathbf{F}_p(\mathbf{u},p)=\mathbf{0},\\
		& \mathbf{u}(a)=\mathbf{u}_0,
	\end{aligned}\right.
\end{equation}
where $\mathbf{F}_\mathbf{u}(\mathbf{u},p)$ is the Jacobian matrix and $\mathbf{F}_p(\mathbf{u},p)$ is the derivative vector with respect to $p$. If $\mathbf{F}_\mathbf{u}(\mathbf{u},p)$ is nonsingular, the solution path $\mathbf{u}(p)$ is smooth and unique. However, when $\mathbf{F}_\mathbf{u}(\mathbf{u},p)$ becomes singular, the solution path yields different types of bifurcations \cite{bates2013numerically}. Then the numerical homotopy tracking could become  very inefficient. In order to solve this numerical issue, a trial-and-error homotopy tracking method \cite{bates2013numerically,wampler2005numerical} and an adaptive homotopy tracking method \cite{HZ} have been developed to control the stepsize of $p$. However, the computational cost could still be very expensive when the homotopy tracking method is applied to the large-scale nonlinear systems due to the slow tracking near the singularity.

To address this challenge, we propose to solve a stochastic version of the Davidenko differential equation by introducing a noise term, namely
\begin{equation}\label{Sys:rand}
	\left\{
	\begin{aligned}
		&	\mathbf{F}_\mathbf{u}\big(\mathbf{u}(p,\omega),p\big)d\mathbf{u}(p,\omega)+\mathbf{F}_p(\mathbf{u}(p,\omega),p)dp=\mathbf{g}(\mathbf{u}(p,\omega),p)dW(p,\omega),\\
		& \mathbf{u}(a,\omega)=\mathbf{u}_0,
	\end{aligned}\right.
\end{equation}
where $\omega$ is a random variable and possesses the initial
condition $\mathbf{u}(a,\omega) = \mathbf{u}_0$ with probability one and
$dW(p,\omega)$ denotes differential form of the Brownian motion \cite{arnold1974stochastic}. Then, in this case, the solution path can avoid the singularity with probability one (See Fig. \ref{fig:illustration} for an illustration).

\begin{figure}[th]
	\centering
	\includegraphics[width=0.45\linewidth]{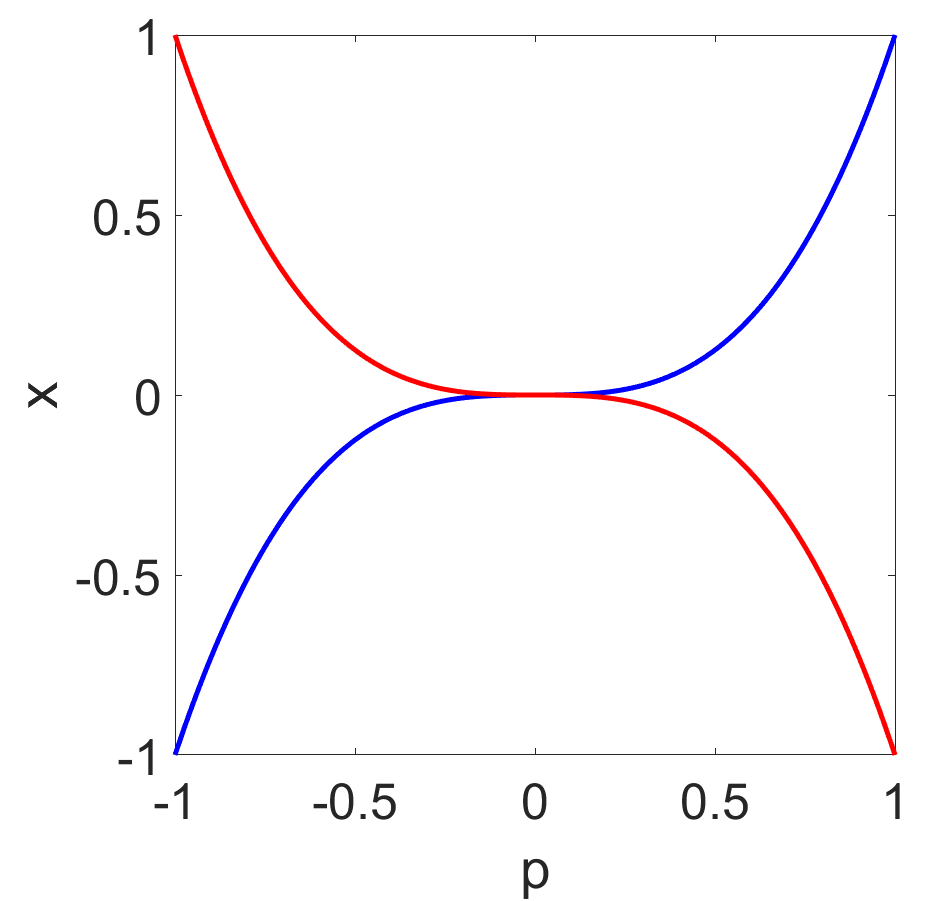}	\includegraphics[width=0.45\linewidth]{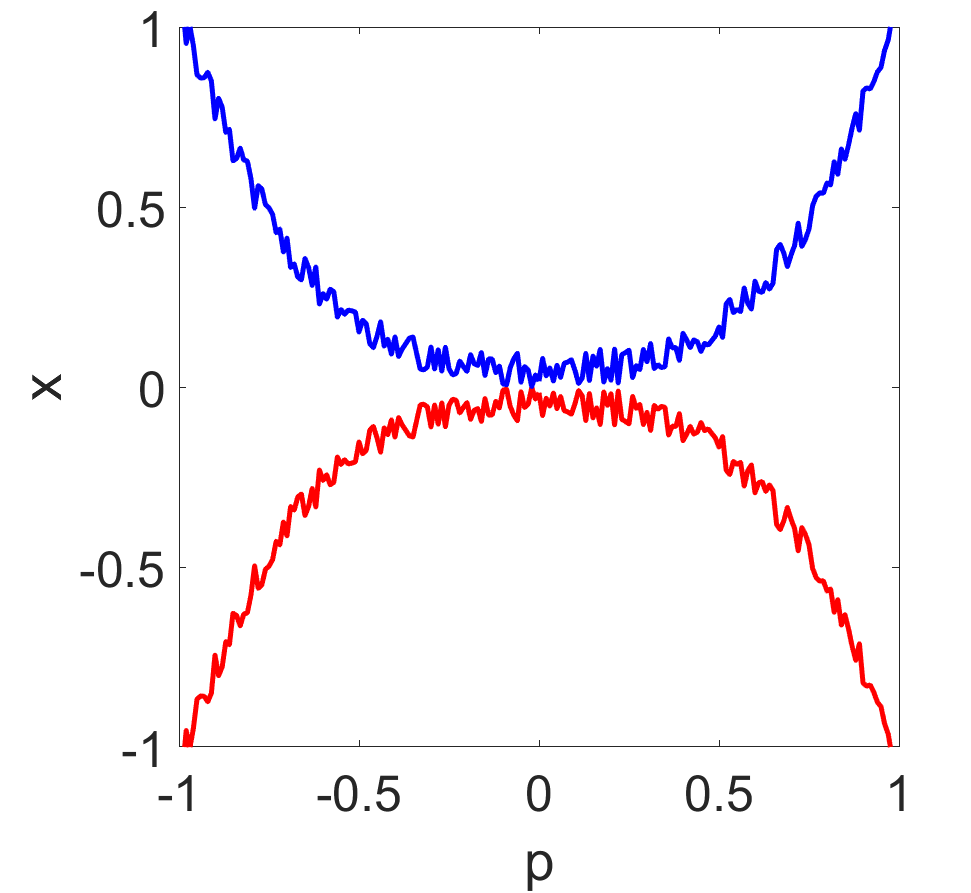}
	\caption{An illustration example, $x^2-p^6=0$, has two solution paths $x=\pm p^3$ and one bifurcation point at $p=0$. The traditional homotopy tracking ({\bf Left}) hits the bifurcation point  while the stochastic tracking ({\bf Right}) can avoid the bifurcation point by tracking $x=\pm (p^3+\xi)$, where $\xi\sim \mathcal{N}(0,0.1)$.}
	\label{fig:illustration}
\end{figure}

In order to integrate the idea of  the stochastic differential equation into the homotopy tracking, we   track the solution $\mathbf{u}(p)$	 from $p=a$ to $p=b$ with a stepsize $\Delta p$. Then for each $p_k=a+\Delta p \cdot k$, we solve the stochastic system below
	\begin{equation}\label{Sys:new} 
		\tilde{\mathbf{F}}(\mathbf{u},p_k) = \begin{bmatrix}
			F_1(\mathbf{u},p_k)\\\vdots\\F_{i-1}(\mathbf{u},p_k)\\\mathbf{u}(j) -\tilde{\mathbf{u}}_{k-1}(j)\\F_{i+1}(\mathbf{u},p_k) \\ \vdots \\F_{n}(\mathbf{u},p_k) \\
		\end{bmatrix}:= \mathbf{F}(\mathbf{u},p_k,\xi = (i,j))=\mathbf{0},
	\end{equation}
	where $\tilde{\mathbf{u}}_{k-1}$ is the solution from previous step and $\tilde{\mathbf{F}}:\mathbb{R}^{n}\times\mathbb{R}\rightarrow\mathbb{R}^{n}$ can be viewed as randomly chooses $n-1$ equations from $\mathbf{F}(\mathbf{u},p_k)$ and replaces $F_i$ by $\mathbf{u}(j) -\tilde{\mathbf{u}}_{k-1}(j)$.
Here the random variable $\xi$ follows the uniform distribution, namely $\mathbb{P}(\xi=(i,j)) = \frac{1}{n^2}$, and quantifies the perturbations to the original system $\mathbf{F}(\mathbf{u},p_k)$. More generally, we can randomly replace $m$ ($1\leq m\leq n$) equations
	of ${\mathbf{F}}(\mathbf{u},p)$ by $\mathbf{u}(\mathcal{J}) =\tilde{\mathbf{u}}_{k-1}(\mathcal{J})$ where $\mathcal{J}=(\mathcal{J}_1,\cdots,\mathcal{J}_m)$ is a $m$ index. Then we define the $s$-th equation of $\tilde{\mathbf{F}}$ as
	\begin{equation}
		\tilde{\mathbf{F}}_s(\mathbf{u},p_k)   = \left\{
		\begin{aligned}
			&\mathbf{F}_s(\mathbf{u},p_k),             \qquad\quad ~~~\qquad s \notin \mathcal{I}\\
			&\mathbf{u}(\mathcal{J}_c)-\tilde{\mathbf{u}}_{k-1}(\mathcal{J}_c),\qquad s\in \mathcal{I} \hbox{~and~} \mathcal{I}_c=s
		\end{aligned}
		\right.,
		\label{Sys:m_general}
	\end{equation}
	where $\mathcal{I}=(\mathcal{I}_1,\cdots,\mathcal{I}_m)$ stands for randomly choosing $m$ equations. If $s\in \mathcal{I}$, then we find $c$ such that $\mathcal{I}_c=s$ and replace the $s$-th equation by the previous value, namely,  $\mathbf{u}(\mathcal{J}_c)-\tilde{\mathbf{u}}_{k-1}(\mathcal{J}_c)$.
	Here $\mathcal{I}$ and $\mathcal{J}$ are randomly drawn from the uniform distribution, namely $\mathbb{P}(\mathcal{I},\mathcal{J}) = \frac{1}{(C_n^m)^2}$.  We denote the set of all possible $m$ indexes as $\mathcal{M}$.

Finally, we summarize the stochastic homotopy tracking algorithm in {\bf Algorithm \ref{alg1}}. In this algorithm, we increase the number of random equations, $m$, if there is no solution to the stochastic system $\tilde{\mathbf{F}}({\mathbf{u}},p_{k+1})=0$. This is equivalent to perform a larger perturbation to the original system by solving fewer equations. Similarly, we could also increase the perturbation by  setting an adaptive tolerance for $\|{\mathbf{F}}(\tilde{\mathbf{u}}_{k+1},p_{k+1})\|<TOL$ by fixing  the number of randomly choosing equation, $m$.

\begin{algorithm}[H]
	\SetAlgoLined
	\KwIn{A step-size $\Delta p$, a threshold $TOL$, and a start point $(\tilde{\mathbf{u}}_0,p_0)$}
	\KwOut{A nearby solution path $(\tilde{\mathbf{u}}_k,p_k)_{k=1}^N$}
	\For{$k=0,\cdots,N$}{
		Set m=1\;
		Randomly choose $n-m$ equations and $n-m$ variables to form the stochastic system $\tilde{\mathbf{F}}(\mathbf{u},p_{k+1})$ \eqref{Sys:new}\;
		Solve $\tilde{\mathbf{F}}(\mathbf{u},p_{k+1}) = \mathbf{0}$ using the predictor-corrector method\;
		\uIf{$\|\tilde{\mathbf{F}}(\tilde{\mathbf{u}}_{k+1},p_{k+1})\| < TOL$}{Update the solution sequence;}\Else{Increase $m$ and solve the stochastic system again.}
	}\caption{The pseudocode of the stochastic homotopy tracking algorithm.}\label{alg1}
\end{algorithm}

\section{Convergence Analysis}

We employ  the Euler predictor and the Newton corrector \cite{allgower2003introduction} for the homotopy tracking algorithm:
{Given a solution $(\mathbf{u}_0,p_0)$ on the path, that
	is, $\mathbf{F}(\mathbf{u}_0,p_0)=0$, an Euler predictor step gives
	\begin{equation}
		\mathbf{F}_\mathbf{u}(\mathbf{u}_0,p_0)\Delta \mathbf{u}=-\mathbf{F}_p(\mathbf{u}_0,p_0)\Delta
		p,\end{equation}
	and then letting ${\mathbf{u}}_1=\mathbf{u}_0+\Delta\mathbf{u}$; The Newton corrector reads \begin{equation}
		\mathbf{F}_\mathbf{u}({\mathbf{u}}_1,p_1)\Delta
		\mathbf{u}=-\mathbf{F}({\mathbf{u}}_1,p_1).\end{equation}
	Then we repeat this correction until $(\mathbf{u}_1,p_1)$ is on the path. The predictor-corrector method  for the stochastic homotopy tracking method needs to replace $\mathbf{F}$ by $\tilde{\mathbf{F}}$ defined in \eqref{Sys:new} with the corresponding derivatives below:
	\begin{equation}
		\begin{aligned}
			&\tilde{\mathbf{F}}_p(\mathbf{u})= \mathbf{F}_p(\mathbf{u},\xi = (i,j)) = \mathbf{F}_p(\mathbf{u}) - \frac{\partial F_i}{\partial p}\mathbf{e}_i,\\
			&\tilde{\mathbf{F}}_\mathbf{u}(\mathbf{u})= \mathbf{F}_\mathbf{u}(\mathbf{u},\xi=(i,j)) = \mathbf{F}_\mathbf{u}(\mathbf{u})  - \mathbf{e}_i\frac{\partial F_i}{\partial\mathbf{u}}(\mathbf{u}) + E_{ij},
		\end{aligned}\nonumber
	\end{equation}
	where $E_{ij}$ is a matrix with all zero elements except the $(i,j)$-th element as one.
For the general stochastic system \eqref{Sys:m_general} with $m$ random equations, we have $\xi = (\mathcal{I},\mathcal{J})$ and
		\begin{equation}
			\begin{aligned}
				&\tilde{\mathbf{F}}_p(\mathbf{u})= \mathbf{F}_p(\mathbf{u}) - \sum_{i\in \mathcal{I}}\frac{\partial F_i}{\partial p}\mathbf{e}_i\triangleq\mathbf{F}_p(\mathbf{u}) - C(\mathbf{u},\xi),\\
				&\tilde{\mathbf{F}}_\mathbf{u}(\mathbf{u}) = \mathbf{F}_\mathbf{u}(\mathbf{u})  - \sum_{i\in \mathcal{I}}\mathbf{e}_i\frac{\partial F_i}{\partial\mathbf{u}}(\mathbf{u}) + \sum_{i\in \mathcal{I},j\in \mathcal{J}}E_{ij} \triangleq \mathbf{F}_\mathbf{u}(\mathbf{u}) - S(\mathbf{u},\xi).
			\end{aligned}\nonumber
		\end{equation}

	We also define the tensor $\nabla\mathbf{F}_\mathbf{u}(\mathbf{u})$ as follows:
	\begin{equation*}
		[\nabla\mathbf{F}_\mathbf{u}(\mathbf{u})]_{ijk} = [\nabla^2 \mathbf{F}_i(\mathbf{u})]_{jk},\quad i,j,k \in\{1,2,\cdots,n\}
	\end{equation*}
	and define the multiplication of the tensor with a vector, $\mathbf{b}\in\mathbb{R}^n$, as
	\begin{equation*}
		[\nabla\mathbf{F}_\mathbf{u}(\mathbf{u})\mathbf{b}]_{ij} = \sum_{k = 1}^{n}[\nabla^2 \mathbf{F}_i(\mathbf{u})]_{jk}\mathbf{b}_k.
	\end{equation*}
	Then $\|\nabla\mathbf{F}_\mathbf{u}(\mathbf{u}) \| = \max_{1\le i\le n} \|\nabla^2 \mathbf{F}_i(\mathbf{u})\|$. In this section, we analyze that the solution path guided by the stochastic homotopy tracking is closed to the path guided by the traditional homotopy tracking under certain conditions. This analysis is performed for Euler's prediction in Theorem 3.1 and for Newton's correction in Theorem 3.2.

	\begin{theorem}[Euler's Prediction]
		Suppose $\mathbf{u}_{0}$ and $\tilde{\mathbf{u}}_{0}$ are the start points for the original system $\mathbf{F}$ and the stochastic system $\tilde{\mathbf{F}}$ respectively.
		If we have the following assumptions
		\begin{itemize}
			\item  $\mathbf{F}_\mathbf{u}$ and $\tilde{\mathbf{F}}_\mathbf{u}$ are invertible and differentiable and \[\|\mathbf{F}_\mathbf{u}\|\leq L_\mathbf{u}, \|\mathbf{F}_\mathbf{u}^{-1}\|\leq M_\mathbf{u} \hbox{~and~}\|\tilde{\mathbf{F}}_\mathbf{u}^{-1}\|\leq M_\mathbf{u};\]
			\item $\nabla\mathbf{F}_\mathbf{u}$, $\nabla\tilde{\mathbf{F}}_\mathbf{u}$ are continuous;
			\item $\mathbf{F}_p$ and $\tilde{\mathbf{F}}_p$ are differentiable and $\|\mathbf{F}_p\|\leq M_p$;
			\item {$\nabla\mathbf{F}_p$ is continuous and $\|\nabla\mathbf{F}_p\|\leq L_p$,}
		\end{itemize}
		then we have
		\begin{equation}
			\begin{aligned}
				\|\mathbb{E}(\mathbf{u}_{N}-\tilde{\mathbf{u}}_{N})\|^2\le &{CS_1}\|\mathbb{E}(\mathbf{u}_{0}-\tilde{\mathbf{u}}_{0})\|^2+{CS_2}\frac{m^2}{n^2} + \mathcal{O}(\frac{m^2\Delta p}{n^2}),
			\end{aligned}
		\end{equation}
		where $CS_1$ and $CS_2$ are constants.	
		\label{thm:euler_m=1}
	\end{theorem}
	
	\begin{proof}
		We compare the predictor step of  the traditional and the stochastic homotopy tracking at $p=p_{k-1}$ and obtain
		\begin{equation}
			\begin{aligned}
				\mathbf{u}_{k} &= \mathbf{u}_{k-1} +\mathbf{F}^{-1}_\mathbf{u}(\mathbf{u}_{k-1})\mathbf{F}_p(\mathbf{u}_{k-1})\Delta p,\\
				\tilde{\mathbf{u}}_{k}&=\tilde{\mathbf{u}}_{k-1} +\tilde{\mathbf{F}}^{-1}_\mathbf{u}(\tilde{\mathbf{u}}_{k-1})\tilde{\mathbf{F}}_p(\tilde{\mathbf{u}}_{k-1})\Delta p,\\
			\end{aligned}
		\end{equation}
		which implies
		\begin{equation}
			\mathbf{u}_{k}  - \tilde{\mathbf{u}}_{k} =  \mathbf{u}_{k-1} - \tilde{\mathbf{u}}_{k-1}  +\mathbf{FP}(\mathbf{u}_k,\tilde{\mathbf{u}}_k)\Delta p,
		\end{equation}
		where $\mathbf{FP}(\mathbf{u}_k,\tilde{\mathbf{u}}_k)=\mathbf{F}^{-1}_\mathbf{u}(\mathbf{u}_{k-1})\mathbf{F}_p(\mathbf{u}_{k-1})-\tilde{\mathbf{F}}^{-1}_\mathbf{u}(\tilde{\mathbf{u}}_{k-1})\tilde{\mathbf{F}}_p(\tilde{\mathbf{u}}_{k-1})$.
		Then by taking the expectation with respect to $\xi$, we have 	
		\begin{equation}
			\begin{aligned}
				&\ \|\mathbb{E}(\mathbf{u}_{k}-\tilde{\mathbf{u}}_{k})\|^2=\|\mathbb{E}( \mathbf{u}_{k-1} - \tilde{\mathbf{u}}_{k-1})  + \mathbb{E}\big(\mathbf{FP}(\mathbf{u}_k,\tilde{\mathbf{u}}_k)\big)\Delta p  \|^2\\
				\le& \|\mathbb{E}( \mathbf{u}_{k-1} - \tilde{\mathbf{u}}_{k-1})\|^2 + \| \mathbb{E}\big(\mathbf{FP}(\mathbf{u}_k,\tilde{\mathbf{u}}_k)\big)  \|^2\Delta p^2+2 \|\mathbb{E}( \mathbf{u}_{k-1} - \tilde{\mathbf{u}}_{k-1})\| \|\mathbb{E}\big(\mathbf{FP}(\mathbf{u}_k,\tilde{\mathbf{u}}_k)\big)  \|\Delta p\\
				\le& ( 1 + \Delta p)\|\mathbb{E}( \mathbf{u}_{k-1} - \tilde{\mathbf{u}}_{k-1})\|^2  +\|\mathbb{E}\big(\mathbf{FP}(\mathbf{u}_k,\tilde{\mathbf{u}}_k)\big) \|^2(\Delta p + \Delta p ^2).	\label{eqn:Euler}
			\end{aligned}
		\end{equation}
		
		Moreover, by Taylor's theorem, there exists $\mathbf{t}_{k-1}$ such that
		\begin{equation}
			\mathbf{F}_\mathbf{u}(\mathbf{u}_{k-1})=\mathbf{F}_\mathbf{u}(\tilde{\mathbf{u}}_{k-1})+\nabla\mathbf{F}_\mathbf{u}(\mathbf{t}_{k-1})\cdot(\mathbf{u}_{k-1}-\tilde{\mathbf{u}}_{k-1}).	
		\end{equation}
		Therefore, we have
		\begin{equation}
			\begin{aligned}
				&\mathbf{FP}(\mathbf{u}_k,\tilde{\mathbf{u}}_k)=\mathbf{F}^{-1}_\mathbf{u}(\mathbf{u}_{k-1})\big[\mathbf{F}_p(\mathbf{u}_{k-1}) - \mathbf{F}_\mathbf{u}(\mathbf{u}_{k-1})\mathbf{F}^{-1}_\mathbf{u}(\tilde{\mathbf{u}}_{k-1},\xi_{k})\mathbf{F}_p(\tilde{\mathbf{u}}_{k-1},\xi_{k})\big]\\
				=&\mathbf{F}^{-1}_\mathbf{u}(\mathbf{u}_{k-1})\Big[\mathbf{F}_p(\mathbf{u}_{k-1}) - \Big((\mathbf{F}_\mathbf{u}+S)(\tilde{\mathbf{u}}_{k-1},\xi_{k})+\nabla\mathbf{F}_\mathbf{u}(\mathbf{t}_{k-1})\cdot(\mathbf{u}_{k-1}-\tilde{\mathbf{u}}_{k-1})\Big)\\
				&\mathbf{F}^{-1}_\mathbf{u}(\tilde{\mathbf{u}}_{k-1},\xi_{k})\big(\mathbf{F}_p(\tilde{\mathbf{u}}_{k-1}) - C(\tilde{\mathbf{u}}_{k-1},\xi_{k})\big)\Big]\\
				=&\mathbf{F}^{-1}_\mathbf{u}(\mathbf{u}_{k-1})[\mathbf{F}_p(\mathbf{u}_{k-1}) -\mathbf{F}_p(\tilde{\mathbf{u}}_{k-1}) + R(\tilde{\mathbf{u}}_{k-1},\mathbf{u}_{k-1},\xi_{k})],
			\end{aligned}\label{eq:FP1}
		\end{equation}
		where 	\begin{equation*}
			\begin{aligned}
				&R(\tilde{\mathbf{u}}_{k-1},\mathbf{u}_{k-1},\xi_{k}) \\
				=& C(\tilde{\mathbf{u}}_{k-1},\xi_{k}) - S(\tilde{\mathbf{u}}_{k-1},\xi_{k})\mathbf{F}^{-1}_\mathbf{u}(\tilde{\mathbf{u}}_{k-1},\xi_{k})(\mathbf{F}_p(\tilde{\mathbf{u}}_{k-1}) - C(\tilde{\mathbf{u}}_{k-1},\xi_{k}))\\	&-\nabla\mathbf{F}_\mathbf{u}(\mathbf{t}_{k-1})\cdot(\mathbf{u}_{k-1}-\tilde{\mathbf{u}}_{k-1})\mathbf{F}^{-1}_\mathbf{u}(\tilde{\mathbf{u}}_{k-1},\xi_{k})(\mathbf{F}_p(\tilde{\mathbf{u}}_{k-1}) - C(\tilde{\mathbf{u}}_{k-1},\xi_{k})).
			\end{aligned}
		\end{equation*}
		Moreover, there exists 	
		$\mathbf{s}_{k-1}$ such that \begin{equation*}\mathbf{F}_p(\mathbf{u}_{k-1})=\mathbf{F}^{-1}_\mathbf{u}(\mathbf{u}_{k-1}) + \nabla \mathbf{F}_p(\mathbf{s}_{k-1})(\mathbf{u}_{k-1}-\tilde{\mathbf{u}}_{k-1}),	\end{equation*}
		then Eq. (\ref{eq:FP1}) becomes
		\begin{equation}
			\begin{aligned}
				&\|\mathbb{E}( \mathbf{FP}(\mathbf{u}_{k-1},\tilde{\mathbf{u}}_{k-1}))  \|^2\\
				=&\|\mathbb{E}(  \mathbf{F}^{-1}_\mathbf{u}(\mathbf{u}_{k-1})\nabla \mathbf{F}_p(\mathbf{s}_{k-1})(\mathbf{u}_{k-1}-\tilde{\mathbf{u}}_{k-1})) +\mathbb{E}(   \mathbf{F}^{-1}_\mathbf{u}(\mathbf{u}_{k-1}) R(\tilde{\mathbf{u}}_{k-1},\mathbf{u}_{k-1},\xi_{k}) )  \|^2\\
				\le&2\|\mathbb{E}(  \mathbf{F}^{-1}_\mathbf{u}(\mathbf{u}_{k-1})\nabla \mathbf{F}_p(\mathbf{s}_{k-1})(\mathbf{u}_{k-1}-\tilde{\mathbf{u}}_{k-1}))\|^2 +2\|\mathbb{E}(   \mathbf{F}^{-1}_\mathbf{u}(\mathbf{u}_{k-1}) R(\tilde{\mathbf{u}}_{k-1},\mathbf{u}_{k-1},\xi_{k}) )  \|^2)\\
				\le &2\|\mathbf{F}^{-1}_\mathbf{u}(\mathbf{u}_{k-1})\|^2\|\nabla \mathbf{F}_p(\cdot)\|^2\|\mathbb{E}( (\mathbf{u}_{k-1}-\tilde{\mathbf{u}}_{k-1}))\|^2  +2\|\mathbf{F}^{-1}_\mathbf{u}(\mathbf{u}_{k-1})\|^2\| \mathbb{E}(R(\tilde{\mathbf{u}}_{k-1},\mathbf{u}_{k-1},\xi_{k}))\|^2
			\end{aligned}\label{Eq:FP}
		\end{equation}
		Since $\mathbf{F}^{-1}_\mathbf{u}$ and {$\nabla\mathbf{F}_p$} are bounded, we have
		\begin{equation}
			\|\mathbb{E}( \mathbf{FP}(\mathbf{u}_{k-1},\tilde{\mathbf{u}}_{k-1}))  \|^2
			\le 2M_\mathbf{u}^2 L_p^2\|\mathbb{E}( (\mathbf{u}_{k-1}-\tilde{\mathbf{u}}_{k-1}))\|^2  +2M_\mathbf{u}^2\| \mathbb{E}(R(\tilde{\mathbf{u}}_{k-1},\mathbf{u}_{k-1},\xi_{k}))\|^2
		\end{equation}
		Next we estimate $R(\tilde{\mathbf{u}}_{k-1},\mathbf{u}_{k-1},\xi_{k}) $:
	
			\begin{equation}
				\small
				\begin{aligned} & \|\mathbb{E}(R(\tilde{\mathbf{u}}_{k-1},\mathbf{u}_{k-1},\xi_{k}))\|^2=\|\mathbb{E}_{\xi_0\xi_1\dots\xi_{k-1}}(\mathbb{E}_{\xi_{k}}R(\tilde{\mathbf{u}}_{k-1},\mathbf{u}_{k-1},\xi_{k}))\|^2\\
					=&\Big\|\mathbb{E}\big(\frac{1}{(C_n^m)^2}\sum_{\mathcal{I},\mathcal{J}\in\mathcal{M}}R(\tilde{\mathbf{u}}_{k-1},\mathbf{u}_{k-1},\xi_k=(\mathcal{I},\mathcal{J}))\big)\Big\|^2\\
					\le&3\Big(\underbrace{\big\|\frac{1}{(C_n^m)^2}\sum_{\mathcal{I},\mathcal{J}\in\mathcal{M}}\mathbb{E}(C(\tilde{\mathbf{u}}_{k-1},\xi_k=(\mathcal{I},\mathcal{J})) )\big\|^2}_{A_1}\\
					+& \underbrace{\big\|\frac{1}{(C_n^m)^2}\sum_{\mathcal{I},\mathcal{J}\in\mathcal{M}}\mathbb{E}(S(\tilde{\mathbf{u}}_{k-1},\xi_k=(\mathcal{I},\mathcal{J}))\mathbf{F}^{-1}_\mathbf{u}(\tilde{\mathbf{u}}_{k-1},\xi_k=(\mathcal{I},\mathcal{J}))(\mathbf{F}_p(\tilde{\mathbf{u}}_{k-1}) - C(\tilde{\mathbf{u}}_{k-1},\xi_k=(\mathcal{I},\mathcal{J}))))\big\|^2}_{A_2}\\ +&\underbrace{\big\|\frac{1}{(C_n^m)^2}\sum_{\mathcal{I},\mathcal{J}\in\mathcal{M}}\mathbb{E}(\nabla\mathbf{F}_\mathbf{u}(\mathbf{t}_{k-1})\cdot(\mathbf{u}_{k-1}-\tilde{\mathbf{u}}_{k-1})\mathbf{F}^{-1}_\mathbf{u}(\tilde{\mathbf{u}}_{k-1},\xi_k=(\mathcal{I},\mathcal{J}))(\mathbf{F}_p(\tilde{\mathbf{u}}_{k-1}) }
					\\&\underbrace{- C(\tilde{\mathbf{u}}_{k-1},\xi_k=(\mathcal{I},\mathcal{J}))))\big\|^2}_{A_3}\Big).\\
				\end{aligned}\label{Eq:R}
			\end{equation}

			Since
			\begin{equation*}
				\sum_{\mathcal{I},\mathcal{J}\in\mathcal{M}}C(\tilde{\mathbf{u}}_{k-1},\xi_k=(\mathcal{I},\mathcal{J}))= \sum_{\mathcal{I},\mathcal{J}\in\mathcal{M}}\sum_{i\in \mathcal{I}}\frac{\partial F_i}{\partial p}\mathbf{e}_i=C_n^mC_{n-1}^{m-1}\mathbf{F}_p(\tilde{\mathbf{u}}_{k-1}),
			\end{equation*}
			we have
			\begin{equation*}
				A_1=\|\frac{C_n^mC_{n-1}^{m-1}}{(C_n^m)^2}\mathbb{E} (\mathbf{F}_p(\tilde{\mathbf{u}}_{k-1}))\|^2 = \|\frac{C_{n-1}^{m-1}}{C_n^m}\mathbb{E} (\mathbf{F}_p(\tilde{\mathbf{u}}_{k-1}))\|^2\leq \frac{m^2}{n^2}M_p^2.
			\end{equation*}
			
			Moreover, we have
			\begin{equation*}\small
				\begin{aligned}
					A_2&\leq \|\frac{1}{(C_n^m)^2}\sum_{\mathcal{I},\mathcal{J}\in\mathcal{M}}\mathbb{E}(S(\tilde{\mathbf{u}}_{k-1},\xi_k=(\mathcal{I},\mathcal{J})) \|^2 \|\mathbf{F}^{-1}_\mathbf{u}(\tilde{\mathbf{u}}_{k-1},\cdot)\|^2\|\mathbf{F}_p(\tilde{\mathbf{u}}_{k-1}) \|^2\\
					&\leq \frac{M_\mathbf{u}^2 M_p^2}{(C_n^m)^4}\big\|\sum_{\mathcal{I},\mathcal{J}\in\mathcal{M}}\mathbb{E}(S(\tilde{\mathbf{u}}_{k-1},\xi_k=(\mathcal{I},\mathcal{J})) \big\|^2,
				\end{aligned}
			\end{equation*}
			By the definition of $S(\tilde{\mathbf{u}}_{k-1},\xi_k=(\mathcal{I},\mathcal{J})$, we have
			\begin{equation*}
				\begin{aligned}
					\sum_{\mathcal{I},\mathcal{J}\in\mathcal{M}}\mathbb{E}(S(\tilde{\mathbf{u}}_{k-1},\xi_k=(\mathcal{I},\mathcal{J}))   &=\sum_{\mathcal{J}\in\mathcal{M}} \mathbb{E}\big(\sum_{\mathcal{I}\in\mathcal{M}}(\sum_{i\in \mathcal{I}}\mathbf{e}_i\frac{\partial F_i}{\partial\mathbf{u}}(\mathbf{u}) - \sum_{i\in \mathcal{I},j\in \mathcal{J}}E_{ij})\big)\\
					&=C_n^mC_{n-1}^{m-1}\mathbb{E}(\mathbf{F}_\mathbf{u}(\tilde{\mathbf{u}}_{k-1}))-C_{n-1}^{m-1}C_{n-1}^{m-1}\mathbf{E},		
				\end{aligned}
			\end{equation*}
			where $\mathbf{E}$ is the all-ones matrix. Therefore
			\begin{equation*}
				A_2\leq \frac{m^2}{n^2}(L_\mathbf{u}+1)^2M_\mathbf{u}^2 M_p^2.
			\end{equation*}
			Similarly, we have
			\begin{equation*}
				\begin{aligned}
					&A_3\leq \big\| \frac{1}{(C_n^m)^2}	\sum_{\mathcal{I},\mathcal{J}\in\mathcal{M}} \nabla\mathbf{F}_\mathbf{u}(\mathbf{t}_{k-1})\|^2\|\mathbb{E}(\mathbf{u}_{k-1}-\tilde{\mathbf{u}}_{k-1})\|^2\|\mathbf{F}^{-1}_\mathbf{u}(\tilde{\mathbf{u}}_{k-1},\cdot)\|^2\|\mathbf{F}_p(\tilde{\mathbf{u}}_{k-1})\big\|^2\\
					&\leq L_\mathbf{u}^2 M_\mathbf{u}^2 M_p^2\|\mathbb{E}(\mathbf{u}_{k-1}-\tilde{\mathbf{u}}_{k-1})\|^2.
				\end{aligned}
			\end{equation*}

			Then Eq. (\ref{Eq:R}) becomes
			\begin{equation}
				\|\mathbb{E}(R(\tilde{\mathbf{u}}_{k-1},\mathbf{u}_{k-1},\xi_{k}))\|^2\leq \frac{m^2}{n^2}C_1 +C_2 \|\mathbb{E}(\mathbf{u}_{k-1}-\tilde{\mathbf{u}}_{k-1})\|^2,
			\end{equation}
			where $C_1=3(M_\mathbf{u}^2 M_p^2(L_\mathbf{u}+1)^2+M_p^2)$ and $C_2=3L_\mathbf{u}^2 M_\mathbf{u}^2 M_p^2$.
			
			Then we get the estimate  below
			\begin{equation*}
				\|\mathbb{E}( \mathbf{FP}(\mathbf{u}_{k-1},\tilde{\mathbf{u}}_{k-1}))  \|^2\le  \frac{m^2}{n^2}M_1 + M_2\|\mathbb{E}(\mathbf{u}_{k-1}-\tilde{\mathbf{u}}_{k-1})\|^2,
			\end{equation*}
			where
			\begin{equation*}
				M_1  = 2M_\mathbf{u}^2C_1 \hbox{~and~}		M_2  = 2M_\mathbf{u}^2L_p^2+C_2.
			\end{equation*}
			
			Plugging the above results into \eqref{eqn:Euler}, we have
			\begin{equation}
				\begin{aligned}
					&\|\mathbb{E}(\mathbf{u}_{k}-\tilde{\mathbf{u}}_{k})\|^2\\
					\le& ( 1 + \Delta p)\|\mathbb{E}(\mathbf{u}_{k-1}-\tilde{\mathbf{u}}_{k-1})\|^2+(\frac{m^2M_1}{n^2} + M_2\|\mathbb{E}(\mathbf{u}_{k-1}-\tilde{\mathbf{u}}_{k-1})\|^2))(\Delta p + \Delta p ^2)\\
					\le & \underbrace{( 1 + \Delta p + M_2(\Delta p + \Delta p ^2))}_{\tilde{M}_1}\|\mathbb{E}(\mathbf{u}_{k-1}-\tilde{\mathbf{u}}_{k-1})\|^2+\underbrace{\frac{m^2}{n^2}M_1(\Delta p + \Delta p ^2)}_{\tilde{M}_2}.
				\end{aligned}
			\end{equation}	
			which implies
			\begin{equation}
				\begin{aligned}
					\|\mathbb{E}(\mathbf{u}_{k}-\tilde{\mathbf{u}}_{k})\|^2\le & \tilde{M}_1\|\mathbb{E}(\mathbf{u}_{k-1}-\tilde{\mathbf{u}}_{k-1})\|^2 + \tilde{M}_2\\
					\le &  \tilde{M}_1^2 \|\mathbb{E}(\mathbf{u}_{k-2} -\tilde{\mathbf{u}}_{k-2})\|^2 +\tilde{M}_1\tilde{M}_2+\tilde{M}_2\\
					\le  & \tilde{M}_1^k\|\mathbb{E}(\mathbf{u}_{0} -\tilde{\mathbf{u}}_{0})\|^2+(1 +\tilde{M}_1 + \cdots + \tilde{M}_1^{k-1} )\tilde{M}_2\\
					= &\tilde{M}_1^k\|\mathbb{E}(\mathbf{u}_{0} -\tilde{\mathbf{u}}_{0})\|^2+\frac{1 - \tilde{M}_1^k }{1-  \tilde{M}_1}\tilde{M}_2.\\
				\end{aligned}\label{eq:estimate}
			\end{equation}
			
			Then we obtain the estimate of $\frac{1 - \tilde{M}_1^k }{1-  \tilde{M}_1}\tilde{M}_2$ as follows
			\begin{equation*}
				\begin{aligned}
					&\frac{1 - \tilde{M}_1^k }{1-  \tilde{M}_1}\tilde{M}_2 \le\frac{e^{(1+M_2)(b-a)}-1 }{(1+M_2)\Delta p + M_2\Delta p ^2} \tilde{M}_2\\
					\le&\frac{e^{(1+M_2)(b-a)}-1 }{(1+M_2)\Delta p} (1 - \frac{M_2}{1+M_2}\Delta p+ \mathcal{O}(\Delta p^2))\frac{m^2M_1}{n^2}(\Delta p + \Delta p ^2)\\
					\le&\frac{m^2M_1}{n^2}\frac{e^{(1+M_2)(b-a)}-1}{1+M_2} + \mathcal{O}(\frac{m^2\Delta p}{n^2})
				\end{aligned}
			\end{equation*}
			
			Thus, Eq. (\ref{eq:estimate}) becomes
			\begin{equation*}
				\|\mathbb{E}(\mathbf{u}_{N}-\tilde{\mathbf{u}}_{N})\|^2\le \underbrace{e^{(1+M_2)(b-a)}}_{CS_1}\|\mathbb{E}(\mathbf{u}_{0}-\tilde{\mathbf{u}}_{0})\|^2+\underbrace{M_1\frac{e^{(1+M_2)(b-a)}-1}{1+M_2}}_{CS_2}\frac{m^2}{n^2} + \mathcal{O}(\frac{m^2\Delta p}{n^2}).
			\end{equation*}
	\end{proof}
	\begin{remark}
		For large-scale nonlinear parametric problems, when $n$ is large, the error caused by the stochastic homotopy tracking becomes very small due to the $O(\frac{1}{n^2})$ estimate for any given $m$. Therefore, the Euler's prediction of the stochastic homotopy tracking stays closed to the prediction by the traditional homotopy tracking.
	\end{remark}

	\begin{theorem}[Newton's correction]
		Suppose $\mathbf{u}_{k}^i $ and $\tilde{\mathbf{u}}_{k}^i $ are $i$-th Newton's iterations for solving $\mathbf{F}(\mathbf{u},p_k) = 0$ and  $\tilde{\mathbf{F}}(\mathbf{u},p_k) = 0$ respectively.
		If we have the following assumptions
		\begin{itemize}
			\item  $\mathbf{F}_\mathbf{u}$ and $\tilde{\mathbf{F}}_\mathbf{u}$ are invertible and differentiable and \[ \|\mathbf{F}_\mathbf{u}^{-1}\|\leq M_\mathbf{u}\hbox{~and~}\|\tilde{\mathbf{F}}_\mathbf{u}^{-1}\|\leq M_\mathbf{u};\]
			
			\item $\nabla\mathbf{F}_\mathbf{u}$, $\nabla\tilde{\mathbf{F}}_\mathbf{u}$ are continuous and \[ \|\nabla\mathbf{F}_\mathbf{u}\|\leq K_\mathbf{u}\hbox{~and~}\|\nabla\tilde{\mathbf{F}}_\mathbf{u}\|\leq K_\mathbf{u};\]
			\item  The initial guesses $\mathbf{u}_{k}^0 $ and $\tilde{\mathbf{u}}_{k}^0 $ are in a small neighborhood of the real solutions $\mathbf{u}_{k}$ and $\tilde{\mathbf{u}}_{k} $,
		\end{itemize}
		then we have
		\begin{equation}
			\lim_{i\to\infty}  \| \mathbb{E}(\mathbf{u}^{i}_{k} -\tilde{\mathbf{u}}^{i}_{k} )\|\le\|\mathbb{E}(\mathbf{u}_{k} -\tilde{\mathbf{u}}_{k} )\|.
		\end{equation}
	\end{theorem}

	\begin{proof}
		We consider the $i$-th iteration of Newton's correction for $\mathbf{F}(\mathbf{u},p_k) = 0$ and $\tilde{\mathbf{F}}(\mathbf{u},p_k) = 0$. There exists $\mathbf{t}_{k}$ and $\tilde{\mathbf{t}}_{k}$  such that
		the following Taylor expansions hold
		\begin{equation*}
			\begin{aligned}
				0 &= \mathbf{F}(\mathbf{u}_{k} ,p_k) = \mathbf{F}(\mathbf{u}^i_{k}  )  + \mathbf{F}_\mathbf{u}(\mathbf{u}^i_{k} ) (\mathbf{u}_{k}  - \mathbf{u}^i_{k}  ) + \frac{1}{2}(\mathbf{u}_{k}  - \mathbf{u}^i_{k}  ) ^T \nabla\mathbf{F}_\mathbf{u}(\mathbf{t}_{k} )(\mathbf{u}_{k}  - \mathbf{u}^i_{k} ), \\
				0 &= \tilde{\mathbf{F}}(\tilde{\mathbf{u}}_{k},p_k) = \tilde{\mathbf{F}}(\tilde{\mathbf{u}}^i_{k} )  + \tilde{\mathbf{F}}_\mathbf{u}(\tilde{\mathbf{u}}^i_{k}) (\tilde{\mathbf{u}}_{k} - \tilde{\mathbf{u}}^i_{k}) + \frac{1}{2}(\tilde{\mathbf{u}}_{k} -\tilde{\mathbf{u}}^i_{k} ) ^T \nabla\tilde{\mathbf{F}}_\mathbf{u}(\tilde{\mathbf{t}}_{k} )(\tilde{\mathbf{u}}_{k}-\tilde{\mathbf{u}}^i_{k}).
			\end{aligned}
		\end{equation*}
		Thus the Newton's schemes are re-written as
		\begin{equation*}
			\begin{aligned}
				\mathbf{u}^{i+1}_{k} &= 	\mathbf{u}^{i}_{k} 	-\mathbf{F}^{-1}_\mathbf{u}(\mathbf{u}^{i}_{k} ) \mathbf{F}(\mathbf{u}^{i}_{k}  ) = \mathbf{u}_{k} + \frac{1}{2}\mathbf{F}^{-1}_\mathbf{u}(\mathbf{u}^{i}_{k} ) (\mathbf{u}_{k}  - \mathbf{u}^{i}_{k}  ) ^T \nabla\mathbf{F}_\mathbf{u}(\mathbf{t}_{k} )(\mathbf{u}_{k}  - \mathbf{u}^{i}_{k} ), \\
				\tilde{\mathbf{u}}^{i+1}_{k} &=	\tilde{\mathbf{u}}^i_{k} - \tilde{\mathbf{F}}^{-1}_\mathbf{u}(\tilde{\mathbf{u}}^i_{k})\tilde{\mathbf{F}}(\tilde{\mathbf{u}}^i_{k} )   =  \tilde{\mathbf{u}}_{k}  + \frac{1}{2}\tilde{\mathbf{F}}^{-1}_\mathbf{u}(\tilde{\mathbf{u}}^i_{k})(\tilde{\mathbf{u}}_{k} -\tilde{\mathbf{u}}^i_{k} ) ^T \nabla\tilde{\mathbf{F}}_\mathbf{u}(\tilde{\mathbf{t}}_{k} )(\tilde{\mathbf{u}}_{k}-\tilde{\mathbf{u}}^i_{k}).
			\end{aligned}
		\end{equation*}
		
		Therefore,
		\begin{equation}
			\label{eqn:newton}
			\begin{aligned}
				\|\mathbb{E}( \mathbf{u}^{i+1}_{k} -\tilde{\mathbf{u}}^{i+1}_{k} )\| =&\Big\|\mathbb{E}\Big( \big((\mathbf{u}^{i}_{k} -\mathbf{F}^{-1}_\mathbf{u}(\mathbf{u}^{i}_{k})\mathbf{F}(\mathbf{u}^{i}_{k}) \big) -\big(\tilde{\mathbf{u}}^{i}_{k} -\tilde{\mathbf{F}}^{-1}_\mathbf{u}(\tilde{\mathbf{u}}^{n}_{k})\tilde{\mathbf{F}}(\tilde{\mathbf{u}}^{n}_{k}))\big)\Big)\Big\|\\
				=&\| \mathbb{E}(\mathbf{u}_{k} -\tilde{\mathbf{u}}_{k} ) + \mathbb{E}(\frac{1}{2}\mathbf{F}^{-1}_\mathbf{u}(\mathbf{u}^{i}_{k} ) (\mathbf{u}_{k}  - \mathbf{u}^{i}_{k}  ) ^T \nabla\mathbf{F}_\mathbf{u}(\mathbf{t}_{k} )(\mathbf{u}_{k}  - \mathbf{u}^{i}_{k} ))\\
				&-\mathbb{E}(\frac{1}{2}\tilde{\mathbf{F}}^{-1}_\mathbf{u}(\tilde{\mathbf{u}}^{i}_{k})(\tilde{\mathbf{u}}_{k} -\tilde{\mathbf{u}}^{i}_{k} ) ^T \nabla\tilde{\mathbf{F}}_\mathbf{u}(\tilde{\mathbf{t}}_{k})(\tilde{\mathbf{u}}_{k}-\tilde{\mathbf{u}}^{i}_{k})  )\|\\
				\le &\|\mathbb{E}(\mathbf{u}_{k} -\tilde{\mathbf{u}}_{k}) \|+ \frac{1}{2}\|\mathbf{F}^{-1}_\mathbf{u}(\mathbf{u}^{i}_{k} )\|\|\nabla\mathbf{F}_\mathbf{u}(\mathbf{t}_{k})\|\mathbb{E}(\|\mathbf{u}_{k}  - \mathbf{u}^{i}_{k} \|^2)\\
				&+\frac{1}{2}\|\tilde{\mathbf{F}}^{-1}_\mathbf{u}(\tilde{\mathbf{u}}^{i}_{k}) \| \|\nabla\tilde{\mathbf{F}}_\mathbf{u}(\tilde{\mathbf{t}}_{k} )\|\mathbb{E}(\|\tilde{\mathbf{u}}_{k}-\tilde{\mathbf{u}}^{i}_{k} \|^2)\\
				\le &\|\mathbb{E}(\mathbf{u}_{k} -\tilde{\mathbf{u}}_{k} )\|  + M_\mathbf{u}{K_\mathbf{u}}\big(\mathbb{E}(\|\mathbf{u}_{k}  - \mathbf{u}^{i}_{k} \|^2) + \mathbb{E}(\|\tilde{\mathbf{u}}_{k}-\tilde{\mathbf{u}}^{i}_{k} \|^2)\big).
			\end{aligned}
		\end{equation}
		
		Due to the local assumption of the initial guesses, then we have the quadratic convergence of Newton's method, namely,
		\begin{equation}
			\begin{aligned}
				\mathbb{E}(\|\mathbf{u}_{k}  - \mathbf{u}^{i}_{k} \| ) &\le \alpha \mathbb{E}(\|\mathbf{u}_{k}  - \mathbf{u}^{i-1}_{k} \| ^2),\\
				\mathbb{E}(\|\tilde{\mathbf{u}}_{k}-\tilde{\mathbf{u}}^{i}_{k} \|) &\le \tilde{\alpha} \mathbb{E}(\|\tilde{\mathbf{u}}_{k}-\tilde{\mathbf{u}}^{i-1}_{k} \|^2).
			\end{aligned}
		\end{equation}

		Therefore  \begin{equation*}
			\begin{aligned}
				\| \mathbb{E}(\mathbf{u}^{i+1}_{k} -\tilde{\mathbf{u}}^{i+1}_{k} )\|\le& \|\mathbb{E}(\mathbf{u}_{k} -\tilde{\mathbf{u}}_{k} )\|  +  M_\mathbf{u}{K_\mathbf{u}}\big(\alpha\mathbb{E}(\|\mathbf{u}_{k}  - \mathbf{u}^{i-1}_{k} \|^4) + \tilde{\alpha}\mathbb{E}(\|\tilde{\mathbf{u}}_{k}-\tilde{\mathbf{u}}^{i}_{k} \|^4)\big)\\
				\le& \|\mathbb{E}(\mathbf{u}_{k} -\tilde{\mathbf{u}}_{k} )\|  +  M_\mathbf{u}{K_\mathbf{u}}\big(\alpha^i\mathbb{E}(\|\mathbf{u}_{k}  - \mathbf{u}^{0}_{k} \|^{2^{i+1}}) + \tilde{\alpha}^n\mathbb{E}(\|\tilde{\mathbf{u}}_{k}-\tilde{\mathbf{u}}^{i}_{k} \|^{2^{i+1}})\big).  \end{aligned}
		\end{equation*}
		By taking the limit on both sides, we have
		\begin{equation*}
			\lim_{i\to\infty}  \| \mathbb{E}(\mathbf{u}^{i}_{k} -\tilde{\mathbf{u}}^{i}_{k} )\|\le\|\mathbb{E}(\mathbf{u}_{k} -\tilde{\mathbf{u}}_{k} )\|.
		\end{equation*}
		
	\end{proof}
	\begin{remark}
		The difference of Newton's corrections between the traditional and the stochastic homotopy tracking is bounded by the difference of the solutions between the original and the stochastic systems which is pretty small for large scale systems. Thus Newton's corrections by two different homotopy tracking algorithms are near each other.
	\end{remark}

	\section{Numerical Examples}
	In this section, we compare the stochastic homotopy tracking with the traditional homotopy tracking on the Matlab platform. We use the stopping criteria of $\Delta p<10^{-7}$ for the traditional homotopy tracking method to detect the bifurcation points.
	\subsection{Example 1}
	We first consider a homotopy setup for solving a system of polynomial equations with the total degree start system, namely,
	\begin{equation}
		\begin{aligned}
			H(x,y,z;t) &=  t\begin{bmatrix}
				x^2 + y^2 + z^2 - 1\\
				x^2 - y^2 - z^2\\
				x + y + z
			\end{bmatrix} + (1-t)\begin{bmatrix}
				x^2 - 1\\
				y^2 - 1\\
				z-1
			\end{bmatrix}
		\end{aligned}=0.
		\label{ex1}
	\end{equation}
	When $t=0$, the solutions of $H(x,y,z;0)=0$ are known explicitly. The solutions of the target system, $H(x,y,z;1)=0$, are revealed by tracking $t$ from $0$ to $1$ on the complex field. There are four  solution paths needed to track from $0$ to $1$ for $\mathbf{u} = [x,y,z]^T$ shown in Fig. \ref{fig:ex1}. The solid lines indicate the solution path of $x(t)$ for the traditional homotopy tracking, while the dashed lines represent the solution paths guided by stochastic homotopy tracking.
	\begin{figure}[!th]
		\centering
		\includegraphics[width=0.7\linewidth]{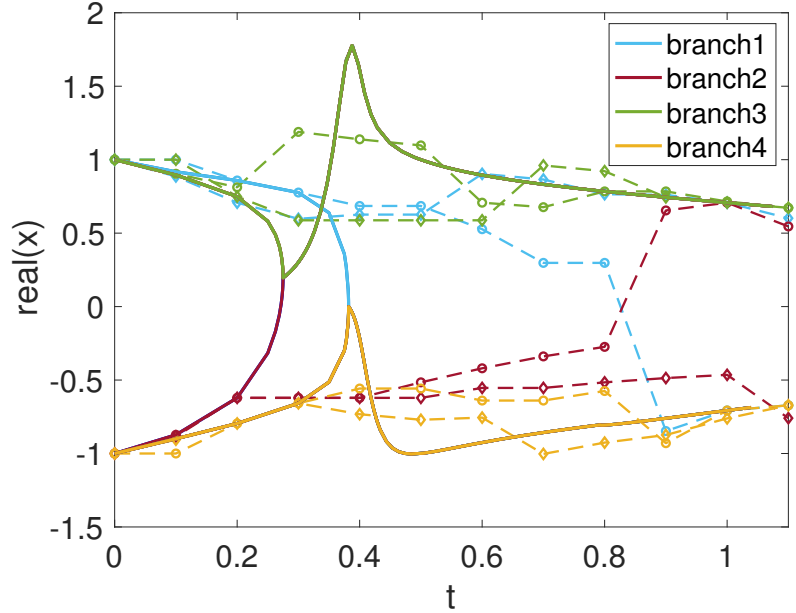}
		\caption{An illustration of the stochastic homotopy tracking method for tracking the solution path $x(t)$ of (\ref{ex1}) on four solution branches. The solid lines are for the traditional homotopy tracking while the dashed lines are for stochastic homotopy tracking.}
		\label{fig:ex1}
	\end{figure}
	The timing data is compared between two tracking methods is shown in Table \ref{tab:ex1} with $\Delta t=0.1$ which clearly demonstrates that the stochastic homotopy tracking method is more efficient with fewer steps from $t = 0$ to $t = 1$.
	\begin{table}
		\begin{center}
			\begin{tabular}[!ht]{|c|c|c|}
				\hline
				&Traditional homotopy tracking	&  Stochastic homotopy tracking \\
				\hline
				Branch 1	&  1.05s (259 steps) &  0.24s (11 steps)\\
				\hline
				Branch 2	& 0.59s (221 steps) & 0.24s (11 steps)\\
				\hline
				Branch 3	& 0.91s (246 steps) & 0.17s (11 steps)\\
				\hline
				Branch 4	& 0.84s (237 steps) & 0.18s (11 steps)\\
				\hline
			\end{tabular}
			\caption{Timing comparison between traditional and stochastic homotopy tracking methods on different branches shown in Fig. \ref{fig:ex1}.}
			\label{tab:ex1}
		\end{center}
	\end{table}

	\subsection{Example 2}
	We consider the following 1D nonlinear boundary value problem.
	\begin{equation}\left\{
		\begin{aligned}
			&u_{xx} = u^2(u^2-p),\\
			&u_x(0) = 0,	u(1) = 0,
		\end{aligned}\right.\label{ex2}
	\end{equation}
	where $p$ is the parameter. The multiple solutions become more as $p$ gets larger. Therefore, turning points happen when $p$ is tracked. We discretize (\ref{ex2}) by using the finite difference method and have the following discretized polynomial system
	\begin{equation}
		\label{eq:discretize ex2}
		\mathbf{F}(\mathbf{u},p):=
		\left(\begin{matrix}
			\frac{1}{h^2}(\mathbf{u}_{1}-2\mathbf{u}_1+\mathbf{u}_{2}) - \mathbf{u}_1^2(\mathbf{u}_1^2-p )\\
			\frac{1}{h^2}(\mathbf{u}_{i-1}-2\mathbf{u}_i+\mathbf{u}_{i+1}) - \mathbf{u}_i^2(\mathbf{u}_i^2-p )\\
			\frac{1}{h^2}(\mathbf{u}_{n-2}-2\mathbf{u}_{n-1}) - \mathbf{u}_{n-1}^2(\mathbf{u}_{n-1}^2-p )\\
		\end{matrix}\right)=0.
	\end{equation}
	where $h=\frac{1}{n}$, $\mathbf{u}\in\mathbb{R}^{n-1}$ and $\mathbf{u}_i = u(\frac{i}{n})$ for $i = 1,2,\cdots,n-1$.
	We track the parameter $p$ from $14$ down to $2$  with $\Delta p=-1$ for one solution path with a turning point shown in Fig~\ref{fig:ex2}. Since the lower solution branch is close to the constant solution branch (the red line in  Fig.~\ref{fig:ex2}, the stochastic homotopy tracking just switches to the constant solution branch when it is close to the turning point.  Moreover, the stochastic homotopy tracking is much efficient than the traditional method by comparing the average tracking time shown in Table \ref{tab:ex2} for different grid points $n$.
	For the upper solution branch, since no nearby solution branch exists, the stochastic homotopy tracking has to deal with a stochastic system with a large perturbation, namely increasing $m$ in {\bf Algorithm \ref{alg1}}.
	
	\begin{figure}[th]
		\centering
		\includegraphics[width=0.95\linewidth]{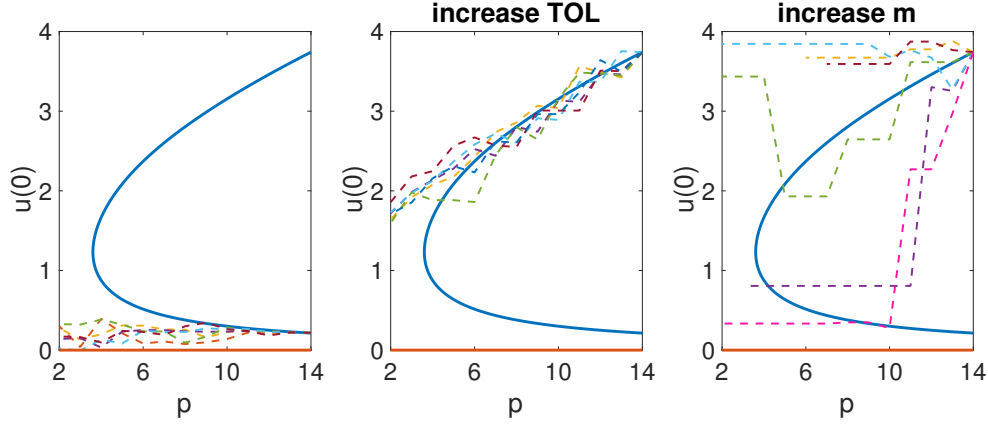}
		\caption{An illustration of stochastic homotopy tracking for tracking (\ref{ex2}) with respect to $p$ from 14 to 2. The lower solution branch is switched to the constant solution branch ({\bf Left}); The upper solution branch needs a large $TOL$ ({\bf Middle}) or a large $m$ ({\bf Right}) in {\bf Algorithm \ref{alg1}}.}
		\label{fig:ex2}
	\end{figure}

	\begin{table}
		\begin{center}
			\begin{tabular}[ht]{|c|c|c|}
				\hline
				n	&Traditional &  Stochastic \\
				\hline
				10	&  0.027s (24 steps) &  0.013s (12 steps)\\
				\hline
				20	&  0.051s (22 steps) &  0.022s (12 steps)\\
				\hline
				40	&  0.141s (30 steps) &  0.076s (12 steps)\\
				\hline
				80	&  0.530s (29 steps) &  0.272s (12 steps)\\\hline
			\end{tabular}
			\caption{Comparison between the traditional and the stochastic homotopy tracking with different number of grid points $n$.}
			\label{tab:ex2}
		\end{center}
	\end{table}

	\subsection{Example 3}
	Last we consider the Schnakenberg model which is a system of partial differential equations shown below \cite{hao2020spatial}:
	\begin{equation}\left\{
		\begin{aligned}
			&\frac{\partial u}{\partial t} = \Delta u + \eta(a-u+u^2v),\\
			&\frac{\partial v}{\partial t} = d\Delta v + \eta(b-u^2v),
		\end{aligned}\right.\label{ex3}
	\end{equation}
	where $u$ is an activator and $v$ is a substrate. The steady-state system of (\ref{ex3}) with non-flux boundary condition has been well-studied in \cite{hao2020spatial} and shown multiple steady-state solutions and the bifurcation structure to the diffusion parameter $d$. In this example, we consider the discretized steady-state system on a 1D domain $x\in[0,1]$ with no-flux boundary conditions:
	\begin{equation}
		\label{eq:discretize ex3}
		\mathbf{F}(\mathbf{u},\mathbf{v},d):=
		\left(\begin{matrix}
			\frac{1}{h^2}(2\mathbf{u}_{2}-2\mathbf{u}_1) +\eta(a-\mathbf{u}_1+\mathbf{u}_1^2\mathbf{v}_1)\\
			\frac{1}{h^2}(\mathbf{u}_{i-1}-2\mathbf{u}_i+\mathbf{u}_{i+1}) +\eta(a-\mathbf{u}_i+\mathbf{u}_i^2\mathbf{v}_i)\\
			\frac{1}{h^2}(2\mathbf{u}_{n}-2\mathbf{u}_{n+1}) +\eta(a-\mathbf{u}_{n+1}+\mathbf{u}_{n+1}^2\mathbf{v}_{n+1})\\
			\frac{d}{h^2}(2\mathbf{v}_{2}-2\mathbf{v}_1) +\eta(b-\mathbf{u}_1^2\mathbf{v}_1)\\
			\frac{d}{h^2}(\mathbf{v}_{i-1}-2\mathbf{v}_i+\mathbf{v}_{i+1}) +\eta(b-\mathbf{u}_i^2\mathbf{v}_i)\\
			\frac{d}{h^2}(2\mathbf{v}_{n}-2\mathbf{v}_{n+1}) +\eta(b-\mathbf{u}_{n+1}^2\mathbf{v}_{n+1})\\
		\end{matrix}\right)=0.
	\end{equation}
	where $h = \frac{1}{n}$, $\mathbf{u},\mathbf{v}\in\mathbb{R}^{n+1}$ with $\mathbf{u}_i = u(\frac{i-1}{n})$ and $\mathbf{v}_i = v(\frac{i-1}{n})$ for $i = 1,2,\cdots,n+1$. We introduce ghost points $\mathbf{u}_{0} ,\mathbf{v}_{0} ,\mathbf{u}_{n+2} $, and $\mathbf{v}_{n+2} $ at $x=0$ and $x=1$. The nonflux boundary conditions imply that $\mathbf{u}_{0} = \mathbf{u}_2$, $\mathbf{v}_{0} = \mathbf{v}_2$,$\mathbf{u}_{n+2} = \mathbf{u}_{n}$, and $\mathbf{v}_{n+2} = \mathbf{v}_{n}$.
	We choose $a = 1/3,\ b = 2/3,\ \eta = 50$ and track $d$ from $50$ to $35$ with different number of grid points $n$. As shown in Fig. \ref{fig:ex3},
	the traditional homotopy tracking method stops near the bifurcation around $d\approx 45$ with a very small tracking stepsize. However, the stochastic homotopy tracking method can avoid the bifurcation point and track down to $35$. Moreover, as $n$ goes larger, the solution path guided by the stochastic homotopy tracking gets closer to the original path.
	Detailed iteration comparison between two tracking methods is shown in Table \ref{tab:ex3_compare} for the different number of grid points $n$ and different tracking stepsizes $\Delta d$. It clearly shows that the stochastic homotopy tracking method becomes more efficient compared to the traditional one as the size of the system gets larger.
	
	\begin{figure}[th]
		\centering
		\includegraphics[width=0.95\linewidth]{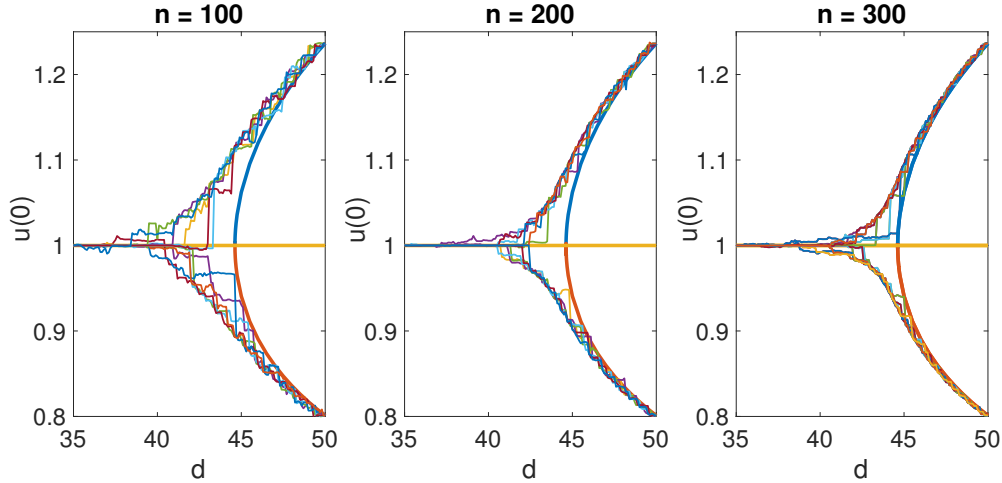}
		\caption{Traditional and stochastic homotopy tracking methods with different number of grid points.}
		\label{fig:ex3}
	\end{figure}

	\begin{table}
		\begin{center}
			\begin{tabular}{|c|c|c|c|c|c|}
				\hline
				\multicolumn{2}{|c|}{\multirow{2}{*}{}}
				&\multicolumn{2}{|c|}{Lower branch}&\multicolumn{2}{|c|}{Upper branch}\\\cline{1-6}
				n & $\Delta d$& Traditional & Stochastic & Traditional  & Stochastic \\
				\hline
				{\multirow{2}{*}{100}}& $-0.5$& 2.76s(32steps)  & 2.30s(31steps) & 2.49s(28steps)  & 1.87s(31steps)  \\\cline{2-6}
				&$-1$	& 3.23s(59steps)  & 1.35s(16steps) & 2.38s(34steps)  & 0.93s(16steps) \\\hline
				{\multirow{2}{*}{200}}& $-0.5$& 12.88s(53steps)   & 8.83s(31steps)  & 10.62s(35steps)  & 8.93s(31steps)\\\cline{2-6}
				&$ -1$	& 9.36s(53steps)   & 3.08s(16steps)  & 7.61s(21steps)  & 3.88s(16steps)\\\hline
				{\multirow{2}{*}{ 300}}& $ -0.5$& 77.9s(90steps)   & 34.1s(31steps)  &  40.2s(34steps)  &  36.9s(31steps) \\\cline{2-6}
				&$ -1$	& 40.3s(90steps)   & 16.5s(16steps)  &  30.1s(34steps)  &  15.6s(16steps)\\\hline
			\end{tabular}
		\end{center}
		\caption{Comparison between traditional and stochastic homotopy tracking with different number of grid points $n$  and  different step-sizes $\Delta d$.}
		\label{tab:ex3_compare}
	\end{table}
	
	\section{Conclusion}
	By taking the path tracking from a stochastic differential equation point of view, we have developed a stochastic homotopy path tracking algorithm that perturbs the  nonlinear parametric system by randomly removing $m$ equation each step. In this paper, we also proved that the solution path guided by the stochastic homotopy algorithm is nearby the original solution path but can avoid the singularities during the tracking. Several numerical examples are used to demonstrate the efficiency of this new method through comparison with the traditional homotopy tracking method. However, the efficiency of the stochastic homotopy tracking depends on the solution landscaping of the original system: if there exists a nearby solution path for bifurcation points, then the stochastic homotopy tracking can switch to the nearby solution paths and keep tracking. Otherwise, the computational cost might be still expensive since it keeps solving stochastic systems by increasing perturbations. In the future, we will improve the efficiency of stochastic homotopy tracking further by exploring the optimal perturbation.

	\section{Data availability}
	Data sharing not applicable to this article as no datasets were generated or analysed during the current study.
	
	\section{Declarations}
	This research is supported by NSF via DMS-1818769. The authors declare that there is no conflict of interest.

\bibliographystyle{plain}
\bibliography{stochastic_homotopy}

\end{document}